\documentclass[12pt]{article}
\usepackage[utf8]{inputenc}
\usepackage{amsmath}
\usepackage{amsthm}
\usepackage{amsfonts}
\usepackage{amssymb}
\usepackage{graphicx}
\usepackage{authblk}

\usepackage{hyperref}
\usepackage{cite}
\usepackage{array}
\usepackage{tabulary}
\usepackage{multicol}
\usepackage{enumerate}
\usepackage{multido}
\usepackage{multirow}
\usepackage{latexsym, amsbsy}

\newtheorem{theorem}{Theorem}[section]

\newtheorem{example}{Example}[section]
\newtheorem{corollary}{Corollary}[section]

\newtheorem{lemma}{Lemma}[section]

\newtheorem{remark}{Remark}[section]

\begin{document}

\title{{\bf Some Results on Number Theory and Analysis}}
\author{B. M. Cerna Magui\~na, Victor H. López Solís 
and 
\\Dik D. Lujerio Garcia 
\\Departamento Acad\'emico de Matem\'atica   
\\ Facultad de Ciencias
\\Universidad Nacional Santiago Ant\'unez de Mayolo
\\Campus Shancayán, Av. Centenario 200, Huaraz, Per\'u 
\\bibianomcm@unasam.edu.pe
\\vlopezs@unasam.edu.pe
\\dlujeriog@unasam.edu.pe}

\maketitle
\begin{center}
In memory of Emilia Maguiña Cabana.
\end{center}
\begin{abstract}
In this paper we obtain bounds for integer solutions of quadratic polynomials in two variables that represent a natural number. Also we get some results on twin prime numbers. In addition, we use linear functionals to prove some results of the mathematical analysis and the Fermat's last theorem.
\end{abstract}
\vspace{0.5 cm}
\section{Introduction}
In \cite{LS} appears a technique to find integer solutions of quadratic polynomials in two variables, the lower bound found in that article was not the best. In this paper we were able to find a lower bound that helps us improve the technique to obtain integer solutions of quadratic polynomials in two variables that represent a natural number. We also get some results on twin prime numbers. We show through a Lemma that the mentioned technique can be used to obtain important results, as is the case in the proofs of the Hölder and Minkowski inequalities in which it is necessary to demonstrate the useful inequality $a^{\lambda}b^{u}\leq \lambda a + u b$ with $\lambda + u =1$ and $a, b, \lambda, u$ are positive numbers. Finally we prove Fermat's last theorem (when $n$ is odd) using elementary calculus techniques.

\section{Bounds for integer solutions of quadratic polynomials in two variables.}

First consider the following general case about bounds.

\begin{theorem}\label{th1}
Let $P$ be a natural number ending in one. If there is $(A,B)\in \mathbb{N}\times \mathbb{N} $ such that:
\begin{description}
\item[(i)] $P=(10A+9)(10B+9)$ or
\item[(ii)] $P=(10A+1)(10B+1)$ or
\item[(iii)] $P=(10A+7)(10B+3).$
\end{description}	
Then we have
\begin{description}
\item for (i) $\dfrac{\sqrt{P}-9}{5}\leq A+B\leq 2 \dfrac{\sqrt{P}-9}{5},  \,\, A<B.$
\item for (ii) $\dfrac{\sqrt{P}-1}{5}\leq A+B\leq 2 \dfrac{\sqrt{P}-1}{5},  \,\, A<B.$
\item for (iii) $\dfrac{\sqrt{P+28}-7}{5}\leq A+B\leq \dfrac{2}{5} \left( \sqrt{P+2}-3 \right),\,\,A<B. $ 
\end{description}											
\end{theorem}
\begin{proof}
For equations $P=(10x+9)(10y+9)$ and $P=(10x+1)(10y+1)$ the proof process is the same. 

If $(A,B)\in \mathbb{N}\times \mathbb{N}$ is a solution of the equation (i), then $(B,A)$ is also a solution of (i). The line $L$ through the points $(A,B)$ and $(B,A)$ is $L:\,\, x+y=A+B$. The line $L_{T}$ through the point $x_{0}=y_{0}=\frac{\sqrt{P}-9}{10}$ and furthermore, as this line being tangent to the curve (i) is given by $L_{T}:x+y=\frac{\sqrt{P}-9}{5}$ (see Figure 1).
\begin{figure}[htb]\label{G1}
\centering
\includegraphics{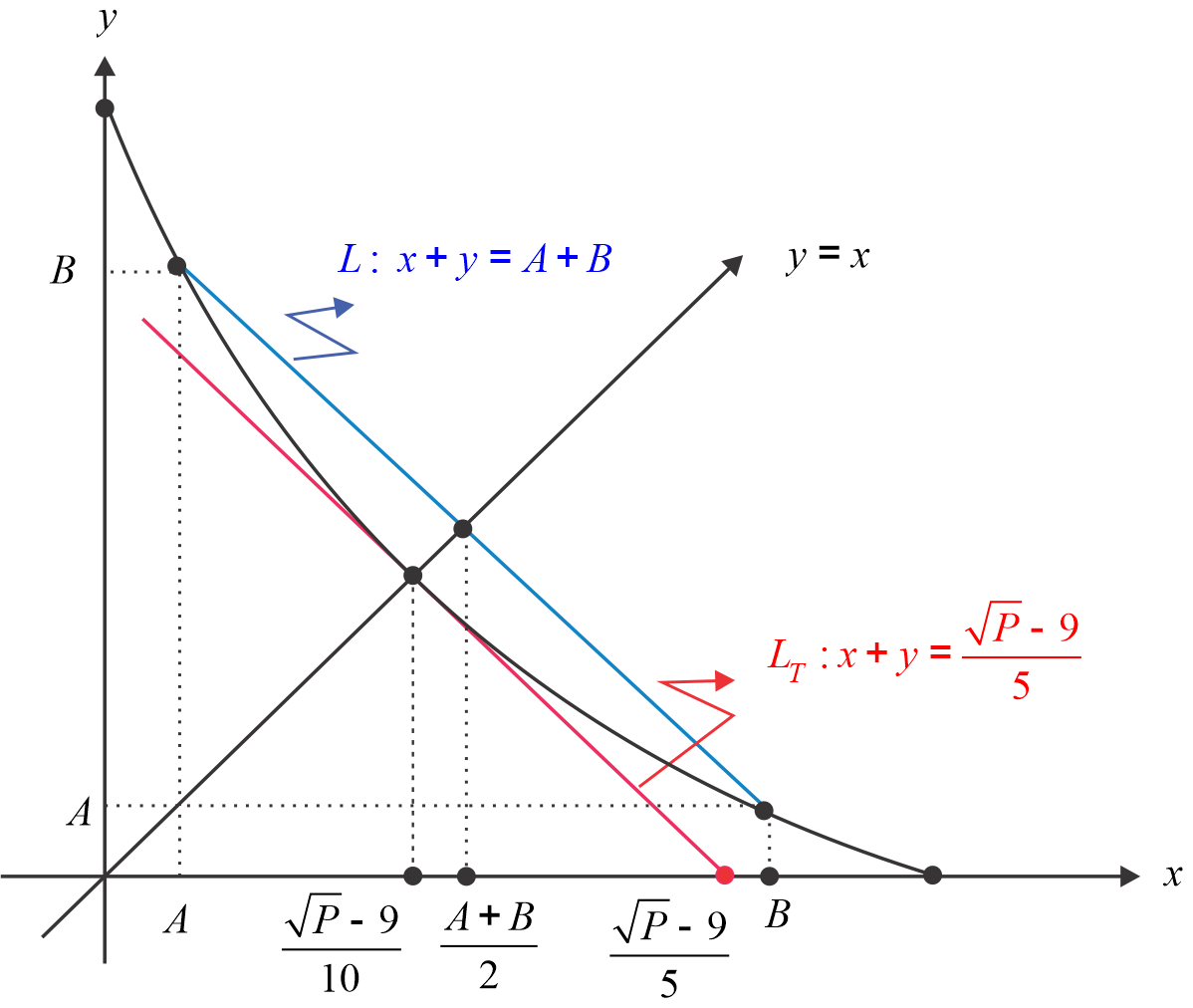}
\caption{The curve $L_{T}.$}
\end{figure}

The proyections of the vector $(x,y)\in L, L_{T}$ when $x=y$ and the vector $(x_{0}, y_{0})$ on the $x$-axis and intersection of the line $L_{T}$ with the $x$-axis given the results. 

For the equation (iii) we have
$$P=(10A+3)(10B+3)+4(10B+3),$$
from this last relationship we have
\begin{eqnarray}\label{eq 1}
q=P-4(10B+3)=(10A+3)(10B+3)
\end{eqnarray}
As it was done in (i) we apply it in (\ref{eq 1}) so
\begin{eqnarray}\label{eq 2}
\dfrac{\sqrt{q}-3}{5}\leq A+B \leq \frac{2}{5} [\sqrt{q}-3]
\end{eqnarray}
By (\ref{eq 1}), $P=q+40B+12$, from this relationship and (\ref{eq 2}) we get
\begin{eqnarray}\label{eq 3}
B\leq \frac{P-21}{50} \,\,\mbox{and}\,\, \frac{2[\sqrt{q}-3]}{5}\leq \frac{2}{5}[\sqrt{P-12}-3]
\end{eqnarray}
Then by (\ref{eq 1}) and (\ref{eq 3}) we have
\begin{eqnarray}\label{eq 4}
\frac{\sqrt{q}-3}{5} \geq \frac{\sqrt{\frac{P+24}{5}}-3}{5}.
\end{eqnarray}
Also of the relations (\ref{eq 2}), (\ref{eq 3}) and (\ref{eq 4}) we have:

\begin{eqnarray} \label{eq 5}
\dfrac{\sqrt{\dfrac{P+24}{5}-3}}{5}\leq A+B\leq \frac{2}{5}\left[ \sqrt{P-12}-3 \right].
\end{eqnarray}
Also the equation (\ref{eq 1}) can be written as
$$P=(10A+7)(10B+7)-4(10A+7)$$
from this relationship we have
\begin{eqnarray}\label{eq 6}
P+4(10A+7)=(10A+7)(10B+7).
\end{eqnarray}
Using relation (i) in (\ref{eq 6}) 
\begin{eqnarray}\label{eqq 7}
\frac{\sqrt{P+28+40A}-7}{5}\leq A+B \leq \frac{2}{5}\left[ \sqrt{P+28+40A}-7 \right]
\end{eqnarray}
of relationships (\ref{eq 5}) and (\ref{eqq 7}) with $A<B$ we get
\begin{eqnarray}\label{eqq8}
\frac{\sqrt{P+28}-7}{5}\leq A+B \leq \frac{2}{5}\left[\sqrt{P-12}-3 \right].
\end{eqnarray}



\end{proof}

\begin{example} \label{EX 1} 
In the theorem \ref{th1} of \cite{LS} we describe an example, where
\begin{description}
\item[a)] $AB=4500+4500\frac{N}{M}-\frac{1}{20}-\frac{N}{20M}$
\item[b)] $A+B=5000-5000 \frac{N}{M}-\frac{1}{18}+\frac{N}{18M}$
\item[c)] $AB+A+B=9500-500\frac{N}{M}-\frac{19}{180}+\frac{N}{150M}$
\end{description}
with $\tau=500\frac{N}{M}+\frac{19}{180}-\frac{1}{180}\frac{N}{M}$ and $\tau$ asume $26$ posibles valores para $A\geq 11$.
From (b) and using (i) we get  
\begin{eqnarray}\label{eq 7}
0,92482 \leq \frac{N}{M}\leq 0,96241.
\end{eqnarray}
Hence replacing $\tau$ in (\ref{eq 7}), we have $462 \leq \tau \leq 481$.
In addition, the following expressions are obtained
\begin{equation}\label{eq 8}
(A-1)(B-1)=-501 +14\tau ; \,\,\, (A+1)(B+1)=501-\tau ; \,\,\, AB=9\tau + 4499
\end{equation}
and from there, for any value of $A=\textup{\r{3}}+2$ or $\textup{\r{3}}+1$ we get $\tau =\textup{\r{3}}$. So $\tau \in \{462,468,471,474, 477, 480 \}$, $\tau$ assumes seven possible values.

As 
\begin{eqnarray}
P&=&\textup{\r{4}}+3=100AB+90(A+B)+ 81=(10A+9)(10B+9) \label{eq 9 }\\
& & \textup{\r{4}}+3=2(A+B)+1, \nonumber
\end{eqnarray}
thus 
$$A+B=\textup{\r{4}}+1 \,\, \mbox{or} \,\, A+B=\textup{\r{4}}+3. $$
Also by (\ref{eq 9 }) we have
\begin{eqnarray}\label{eq 10 }
2A=  \textup{\r{4}} \,\,\, \mbox{and}\,\, 2B=\textup{\r{4}}+2 \,\,\, \mbox{or} \,\,\, 2A=\textup{\r{4}}+2 \mbox{ and } \,\,\, 2B=\textup{\r{4}}
\end{eqnarray}
and by (\ref{eq 10 }) and (\ref{eq 8}) 
\begin{eqnarray}
(2A-2)(B-1)&=& -501\times 2+38 \tau  \nonumber \\\
(\textup{\r{4}}-2)(B-1)&=& 2\tau - (\textup{\r{4}}+2) \nonumber\\
\textup{\r{4}}-2B+2 &=& 2\tau \label{eq 11}
\end{eqnarray}
where the result is equal when $2A=\textup{\r{4}} +2$ y $2B=\textup{\r{4}}$. \newline
By (\ref{eq 11}) it is clear that $\tau =\textup{\r{4}}+1$ or $\tau=\textup{\r{4}}+3$. Therefore $\tau \in \{465,471,477\}$.

If $A+B=4k_{2}+1$, then from (\ref{eq 8}) we know that $\tau= \textup{\r{3}}$ and by (b) $A+B=5001-10\tau$, so $A+B=\textup{\r{3}}$, that is,  $A+B=3k_{1}=4k_{2}+1=5001-10\tau$, which forces us to have $\tau=\textup{\r{6}}$ which is false.
Therefore 
\begin{eqnarray}\label{eq 12}
A+B=4k+3
\end{eqnarray}
and by (\ref{eq 12}) the following possibilities are obtained
\begin{eqnarray} \label{eq 13}
\tau =12 \lambda +9 \,\,\,\,\, \mbox{or}\,\,\,\,\,\tau= 12 \lambda+3
\end{eqnarray}
but $A=3a+1$, $B=3b+2$ or $A=3a+2$, $B=3a+1$, replacing these relations in (\ref{eq 8}) and (\ref{eq 13}) we get
\begin{equation*}\label{eq 14}
3a(3b+1)=-501+3(4\lambda +3)\cdot 19 \,\, \mbox{and}\,\, (3a+2)(3b+3)=9501-(12\lambda +9).
\end{equation*}
For $A=3a+1$, $B=3b+2$, $2A=\textup{\r{4}}$ and $2B=\textup{\r{4}}+2$ we get $a=\textup{\r{4}}$ and $b=\textup{\r{4}}$ or $a=\textup{\r{4}}+1$ and $b=\textup{\r{4}}+3$, by replacing these values in (\ref{eq 14}) we get a contradiction.

The same results are obtained for the other cases. Therefore, the only possibility that guarantees a solution is when $\tau=12\lambda +3$ that is $\tau \in \{471\}$.
\end{example}
Using the ideas of the Theorem \ref{th1} of \cite{LS} we will demonstrate a very important lemma which serves to demonstrate the Hölder inequality and consequently the Minkowsky inequality.
\begin{lemma}\label{Lema1}
If $a,b,\lambda$ and $u$ are non-negative numbers and $u+\lambda =1$, then $a^{\lambda}b^{u}\leq \lambda a+ub$.
\end{lemma}
\begin{proof}
Let $F(x,y)=(a^{\lambda}b^{u})x+ (\lambda a + u b)y$ be so it is clear that $F$ is a continuous linear functional, then $\mbox{Ker} F = (-(\lambda a +ub), a^{\lambda}b^{u})$, $\{\mbox{ Ker} F \} ^{\perp}=(a^{\lambda}b^{u}, \lambda a +ub)$.
Thus
\begin{eqnarray}\label{eq 15}
F(1,1)=a^{\lambda}b^{u}+\lambda a +ub.
\end{eqnarray}
As $\mathbb{R}^{2}=\mbox{Ker} F \oplus \{\mbox{ Ker} F \} ^{\perp}$ we have 
\begin{eqnarray}\label{eq 16}
(1,1)=\lambda_{1}\left(-(\lambda a +ub), a^{\lambda}b^{u} \right)+\lambda_{2}\left(a^{\lambda}b^{u}, \lambda a+ u b\right)
\end{eqnarray}
and relationships (\ref{eq 15}) and (\ref{eq 16}) we get
\begin{eqnarray}\label{eq 17}
a^{\lambda}b^{u} +\lambda a + ub=\lambda_{2} \left(a^{2\lambda}b^{2u}+(\lambda a +ub)^{2} \right)
\end{eqnarray}
and
\begin{eqnarray}\label{eq 18}
a^{\lambda}b^{u}[1-\lambda_{2}a^{\lambda}b^{u}]=(\lambda a +ub) \left[ \lambda _{2}(\lambda a +ub) -1 \right]
\end{eqnarray}
which implies
\begin{eqnarray}\label{eq 19}
1-\lambda _{2} a^{\lambda}b^{u}\geq 0 \,\,\,\,\,\,\mbox{and} \,\,\,\,\,\, \lambda _{2}(\lambda a +ub)-1\geq 0
\end{eqnarray}
or
\begin{eqnarray}\label{eq 20}
1-\lambda _{2} a^{\lambda}b^{u}\leq 0 \,\,\,\,\,\,\mbox{and} \,\,\,\,\,\, \lambda _{2}(\lambda a +ub)-1 \leq 0.
\end{eqnarray}
By (\ref{eq 19}) we get
\begin{eqnarray}\label{eq 21}
\lambda a +\lambda b \geq a^{\lambda}b^{u}
\end{eqnarray} 
and by (\ref{eq 20}) 
\begin{eqnarray}\label{eq 22}
\lambda a + u b\leq a^{\lambda}b^{u}.
\end{eqnarray}
We will show that $(\ref{eq 22})$ not happenes. By (\ref{eq 22}) and (\ref{eq 18}) 
\begin{eqnarray}\label{eq 23}
\lambda _{2}\left[ a^{\lambda}b^{u}+\lambda a+ub \right] \leq 2 
\end{eqnarray}

If $a,b\neq 0$. Thus by (\ref{eq 23}) and using $\lambda +u=1$
\begin{equation}\label{eq 24}
\lambda _{2}\leq \frac{2}{a} \,\, \mbox{if} \,\, a\leq b \,\,\,\,\, \mbox{or}\,\,\,\,\, \lambda_{2}\leq \frac{2}{b} \,\, \mbox{if} \,\,b\leq a
\end{equation}
making $a$ or $b$ big enough, we get the only chance  $\lambda _{2}=0$.

Therefore (\ref{eq 22}) is impossible, so only the relationship happens (\ref{eq 21}).
\end{proof}

\begin{corollary}
Let $H$ be a Pre-Hilbert space over $\mathbb{C}$, then
\begin{center}
  $  |<x,y>|\leq \sqrt{<x,x>}\sqrt{<y,y>}.$
\end{center}
\end{corollary}

Now we will use the Theorem $\ref{th1}$ to study the twin prime numbers. First we have a general result.
\begin{theorem}\label{th2}
Let $P$ be a natural number ending in $1$ and $P=\textup{\r{3}}+2$. If \textbf{(i)}  $P=(10x+9)(10y+9)$ or \textbf{(ii)} $P=(10x+1)(10y+1)$ or \textbf{(iii)} $P=(10x+7)(10y+3)$ and \textbf{(iv)} $P+2=(10x+9)(10y+7)$ or \textbf{(v)} $P+2=(10x+1)(10y+3)$. If there are integer solutions $(A,B)\in \mathbb{N}\times \mathbb{N}$ and $(C,D)\in \mathbb{N}\times \mathbb{N}$ of the quadratic equations representing $P$ and $P+2$, then \newline

\textbf{(i)} $ \dfrac{\sqrt{P}-9}{5}\leq A+B\leq \dfrac{2(\sqrt{P}-9)}{5}$ or \newline

\textbf{(ii)} $\dfrac{\sqrt{P}-1}{5}\leq A+B\leq \dfrac{2(\sqrt{P}-1)}{5}$ or \newline

\textbf{(iii)} $\dfrac{\sqrt{P+28}-7}{5}\leq A +B\leq \frac{2}{5}\left(\sqrt{P-12}-3 \right)$ with $A<B$ and \newline 

\textbf{(iv)} $\dfrac{\sqrt{P+20}-9}{5}\leq C +D\leq \frac{2}{5}\left(\sqrt{P-12}-7 \right)$ with $C<D$. or \newline

\textbf{(v)} $\dfrac{\sqrt{P+8}-3}{5}\leq C +D\leq \frac{2}{5}\left(\sqrt{P}-1 \right)$ with $C<D$. \newline

\end{theorem}

\begin{proof}
The results \textbf{(i)}, \textbf{(ii)}, \textbf{(iii)}, \textbf{(iv)} and  \textbf{(v)} are obtained in a similar way to what was done in the Theorem \ref{th1}. 
\end{proof}

\begin{remark}
From the Theorem \ref{th2} is easy to obtain bounds for $A, B, C, D, AM$ and $CD$. To find the integer solutions it is necessary to use the Theorem \ref{th1} of \cite{LS}, where
\begin{eqnarray*}
AB=\frac{(P-81)}{200}\frac{(M+N)}{M},& &\quad A+B=\frac{(P-81)}{180}\frac{(M-N)}{M}\\
CD=\frac{(P-61)}{200}\frac{(M_{1}+N_{1})}{M_{1}},& & \quad 7C+9D=\frac{(P-61)}{20}\frac{(M_{1}-N_{1})}{N_{1}}.
\end{eqnarray*}
with $A<B$, $C<D$, $N<M$, $M_{1}<N_{1}$, $N$ and $M$ are relative primes,  $N_{1}$ and $M_{1}$ are relative primes.

Also for any $A,B,C,D\in \mathbb{N}$ we have $A+B=\textup{\r{3}}$ and $C+D=\textup{\r{3}}$, and we can use the technique of the Example \ref{EX 1}.

\end{remark}

\begin{theorem}
Let $P$ be a natural number with $P=\textup{\r{3}}+2$. If $P$ and $P+2$ are prime numbers, then exist $m$ and $n$ relatively prime such that $p+1=m=\sqrt{n-1}$ and if $n-2=p_{1}p_{2}$ where $p_{1}$ and $p_{2}$ are prime numbers, then $P$ and $P+2$ are prime numbers.
\end{theorem}
\begin{proof}
Consider the aplication $F(x,y)=Px+(P+2)y$, similar to the Lemma \ref{Lema1} we have
$\textup{Ker}\, F= \{(-(P+2),P)\}$, $\{\textup{Ker}\, F\}^{\perp}=\{(P, P+2)\}$. Then 
\begin{equation} \label{eqw 28}
F(1,1)=P+P+2.
\end{equation}
Also
\begin{equation}\label{eqw 29}
(1,1)=\lambda _{1} \left( -(P+2),P\right)+\lambda_{2}\left(P,P+2 \right)
\end{equation}
and by (\ref{eqw 28}) and (\ref{eqw 29}) we have
\begin{equation}\label{eqw 30}
P+2+P=\lambda_{2}\left(P^{2}+(P+2)^{2} \right).
\end{equation}
Hence $\lambda _{2}\in \mathbb{Q}$. Let $\lambda _{2}=\frac{m}{n}$ be with $m$ and $n$ relatively prime. By (\ref{eqw 30}) we get
\begin{equation}\label{eqw 31}
(P+2)\left(m(P+2)-n \right)=P(n-mP)
\end{equation}
and using the fact that $P+2$ and $P$ are primes, then $P+1=m=\sqrt{n-1}$.

If $m=p+1=\sqrt{n-1}$, then $P^2+2p+1=n-1$ and so $P(P+2)=n-2=p_{1}p_{2}$. 
Therefore $P$ and $P+2$ are prime numbers.
\end{proof}
In the following theorem we prove the Fermat’s last theorem using basic tools developed in this article.

\begin{theorem}\label{th4}
If let $A,B,C\in \mathbb{N}$ be are relatively prime numbers to each other. For $n\in\mathbb{N}$ odd, $n\geq 3$, the equation $A^{n}+B^{n}=C^{n}$ has no integer solution. 															
\end{theorem}
\begin{proof}
Let $F(x,y,z)=A^{n}x+B^{n}y+C^{n}z$ be a real function of several variables, then consider $$ KerF =\{(-C^{n}, 0, A^{n}),(-B^{n}, A^{n}, 0)\}, \{Ker F\}^{\perp}=\left\{\left(A^{n}, B^{n}, C^{n}\right)\right\}.$$
Supuse that equation 
\begin{equation}\label{e1}
    A^{n}+B^{n}=C^{n}
\end{equation}
has integer solution, then 
\begin{eqnarray}\label{eq*1}
F(1,1,1)=2C^{n}.
\end{eqnarray}
Also $(1,1,1)=\lambda_{1}(-C^{n},0,A^{n})+\lambda_{2}(-B^{n},A^{n},0)+\lambda_{3}(A^{n}, B^{n}, C^{n})$. 
Hence, applying F we obtain
\begin{eqnarray}\label{eq*2}
2C^{n}=\lambda_{3}(A^{2n}+B^{2n}+C^{2n})
\end{eqnarray}
From (\ref{eq*2}) 
\begin{eqnarray}\label{eq*3}
\left\{ \begin{array}{l}
A^{n}=\frac{1}{\lambda_{3}}\sin \phi \cos \theta \\
B^{n}=\frac{1}{\lambda_{3}}\sin \phi \sin \theta \\
C^{n}=\frac{1}{\lambda_{3}}(\cos \phi +1)
\end{array}
\right.
\end{eqnarray}
So by (\ref{eq*3}) and using $(\ref{e1})$ we get

\begin{eqnarray}\label{eq*4}
\left(\cos \theta + \sin \theta  \right)=\frac{(\cos \phi +1)}{\sin \phi}=K
\end{eqnarray}
It is clear that $\theta=\theta (A,B,C,n)$, $\phi=\phi (A,B,C,n)$, $\lambda_{3}=\lambda_{3} (A,B,C,n)$ and $K=K(A,B,C,n)$, $\theta \in \langle 0, \frac{\pi}{2}\rangle$, $\phi \in \langle 0, \pi \rangle$.

From (\ref{eq*4})  $\sin \theta = \sqrt{\dfrac{1-K\sqrt{2-K^{2}}}{2}}$, $\cos \theta =\sqrt{\dfrac{1+K\sqrt{2-K^{2}}}{2}} $, $\cos \phi =\dfrac{K^{2}-1}{K^{2}+1}$ and $\sin \phi = \dfrac{2K}{K^{2}+1}$.
Also from the third equation of (\ref{eq*3}) it follows that $\cos \phi \in \mathbb{Q}$ because $\lambda_{3}$ is rational number, so $K^{2}\in \mathbb{Q}, \,\, 1\leq K \leq \sqrt{2}$. 
As $\sin^{2} \theta\in\mathbb{Q}$ making $\lambda K=\pm\sqrt{2-K^{2}}$ with $\lambda\in\mathbb{Q}$, then
$K^{2}=\dfrac{2}{\lambda^{2}+1}$ and $0\leq\lambda^{2}\leq 1$ imply that $-1\leq \lambda \leq 1$. 

As $\lambda_{3}$ and $\lambda$ are rational numbers, consider $\lambda_{3}=\dfrac{M}{N}$ and $\lambda=\frac{r}{t}$ with $G.C.D(M,N)=1$ and $G.C.D(r,t)=1$. 
As $-1\leq \lambda \leq 1$, using the relations of (\ref{eq*3}), if we consider $0\leq\lambda\leq 1$, we obtain
\begin{eqnarray}\label{eq*6}
C^{n}=\frac{4 N t^{2}}{M\left(3 t^{2}+r^{2}\right)}, \quad B^{n}=\frac{2 \operatorname{tN}(t-r)}{M\left(3 t^{2}+r^{2}\right)}, \quad A^{n}=\frac{2 N t(t+r)}{M\left(3 t^{2}+r^{2}\right)}.
\end{eqnarray}
Also, if we consider the case $-1\leq\lambda\leq 0$, we obtain analogous expressions.
Without loss of generality, which will be supported later, from the equation $(\ref{e1})$, if we assume that $C$ is even, then $B$ and $A$ are odd.

We observe from (\ref{eq*6}) that if the factor $\dfrac{2 N t}{M\left(3 t^{2}+r^{2}\right)}$ is a natural number greater than one, $A,B$ and $C$ have common prime factors, which is a contradiction. Then consider
\begin{eqnarray}\label{eq*7}
\frac{1}{\beta}=\frac{2 N t}{M\left(3 t^{2}+r^{2}\right)}, \,\, \beta \in \mathbb{N}, \,\, \beta \geq 1.
\end{eqnarray}
From (\ref{eq*7}) and (\ref{eq*6}) we have
\begin{eqnarray}\label{eq*8}
C^{n}=\dfrac{1}{\beta} 2 t,\quad B^{n}=\frac{1}{\beta}(t-r),\quad A^{n}=\frac{1}{\beta}(t+r).
\end{eqnarray}
If $p\neq 2$ was a prime divisor of the numbers $\beta$, $t-r$ y $t+r$, then $t-r=p\theta_{1}$, $t+r=p\theta_{2}$ and so $2t=p(\theta_{1}+\theta_{2})$, and as $p\neq 2$ implies that $p$ divide to $t$, which implies that $p$ divide to $r$, which is a contradiction, since $G.C.D(r,t)=1$.
Therefore, the unique prime divisor of $\beta$ is 2, then
\begin{eqnarray}\label{eq*9}
C^{n}=t, \quad B^{n}=\dfrac{t-r}{2}, \quad A^{n}=\dfrac{t+r}{2}
\end{eqnarray}
As $C$ is even, $t$ is even, so $r$ is odd.
Consequently, of (\ref{eq*7}) 
\begin{eqnarray}\label{eq*10}
1=\frac{4 N t}{M\left(3 t^{2}+r^{2}\right)}
\end{eqnarray}
which implies that (\ref{eq*9}) y (\ref{eq*10}) are situations that will not happen.
Therefore, $\beta=1$ and so
\begin{eqnarray}\label{eq*11}
C^{n}=2 t,\,\,\, B^{n}=t-r,\,\,\,  A^{n}=t+r. 
\end{eqnarray}
If $t$ was odd, (\ref{eq*11}) would be a contradiction, since 2 does not have an exact nth root and $C\in \mathbb{I}$, which is a contradiction.
Therefore, $t$ is even and $r$ is odd.
Then $t$ has the following form
\begin{eqnarray}\label{eq*12}
t=2^{n-1} p_{1}^{n \alpha_{1}} \ldots p_{k}^{n \alpha_{k}}, \,\,\, n\geq 2.
\end{eqnarray}
In addition from (\ref{eq*11}) we have $B^{n}-A^{n}=2r$. 
Now if $n$ was 2, we would have $(B-A)(B+A)=2r$, but $B-A$ and $B+A$ are pairs, which is absurd. Also, is easy to see that for the other even values of $n$ the same contradiction is reached.
Therefore $n$ has to be odd.

In addition of (\ref{eq*7}) we get
\begin{eqnarray}\label{eq*13}
\frac{2Nt}{M(3t^{2}+r^{2})}=1, \,\,\,\mbox{for all}\, \,\, N,M,t,r.
\end{eqnarray}
Now, we will justify why $C$ was supposed to be even. If $C$ is odd, then $A$ is even and $B$ is odd. Define $-A=\widehat{A}$, $-C=\widehat{C}$, $B=\widehat{B}$.
Thus, we consider
$$G(x,y,z)=\widehat{C}^{n}x+\widehat{B}^{n}y+\widehat{A}^{n}z,\,\, \,\,\, \widehat{C}^{n}+\widehat{B}^{n}=\widehat{A}^{n}.$$ Now, doing the same process done with the linear functional $F$, we get
\begin{eqnarray}\label{eq*14}
\widehat{A}^{n}=\frac{1}{\widehat{\lambda}_{3}} \left( \cos \widehat{\phi} +1\right), \,\,\widehat{B}^{n}=\frac{1}{\widehat{\lambda}_{3}} \sin \widehat{\theta} \sin \widehat{\phi}, \,\, \widehat{C}^{n}=\frac{1}{\widehat{\lambda}_{3}} \sin \widehat{\phi} \cos \widehat{\theta},
\end{eqnarray}
where $\widehat{\theta}\in \langle \frac{3\pi}{2}, 2\pi \rangle$ y $\widehat{\phi} \in \langle 0,\pi \rangle$. We also get
$$\dfrac{\cos \widehat{\phi}+1}{\sin \widehat{\phi}}=\cos \widehat{\theta}+\sin \widehat{\theta} = K_{1}, \,\, K_{1}>0$$
doing similar operations as in the first case, we obtain
\begin{eqnarray}\label{eq*15}
A^{n}=2t_{1}, \,\, B^{n}=r_{1}-t_{1},\,\, C^{n}=r_{1}+t_{1}
\end{eqnarray}
or
\begin{eqnarray}\label{eq*16}
A^{n}=t_{1},\,\,\, B^{n}=\frac{r_{1}-t_{1}}{2}, \,\,\, C^{n}=\frac{r_{1}+t_{1}}{2}
\end{eqnarray}
where $K_{1}^{2}=\dfrac{2}{\lambda _{1}^{2}+1}$ and $\lambda_{1}=\dfrac{r_{1}}{t_{1}}>1$, with $G.C.D(r_{1},t_{1})=1$.
Therefore, (\ref{eq*16}) is a contradiction, because $t_{1}$ is even and $r_{1}$ is odd. Then it is justified that this process is only valid if $n$ is odd.

From the relation (\ref{eq*13}) and the fact $G.C.D(M,N)=G.C.D(t,r)=1$, we conclude that
\begin{eqnarray}\label{eq*17}
N=3 t^{2}+r^{2} \,\,\,\, \mbox{and}  \,\,\,\, M=2t.
\end{eqnarray}
Thus, from the third equation of (\ref{eq*3}) 
\begin{eqnarray}\label{eq*18}
C^{n}=2 \frac{N}{M} \cos ^{2} \frac{\phi}{2}.
\end{eqnarray}
Then by (\ref{eq*18}) and using the fact that $C^{n}=2t$ and by (\ref{eq*17}) we get
\begin{eqnarray}\label{eq*19}
t=\frac{\sqrt{N}}{\sqrt{2}} \cos \frac{\phi}{2}=\frac{\sqrt{N}}{2} \sqrt{\cos \phi+1}.
\end{eqnarray}
So by (\ref{eq*19}) and (\ref{eq*4}) 
\begin{eqnarray}\label{eq*20}
t=\frac{\sqrt{N}}{2} \frac{\sqrt{2} K}{\sqrt{K^{2}+1}}, \,\, \mbox{with} \,\, K^{2}=\frac{2}{\lambda^{2}+1} \,\, \mbox{and}\,\, \lambda=\frac{r}{t}.
\end{eqnarray}
At once, by replacing the value of $K$ we have
\begin{eqnarray}\label{eq*21}
t=\dfrac{\sqrt{2}}{2}K\sqrt{r^{2}+t^{2}}.
\end{eqnarray}
If $r^{2}+t^{2}$ was a perfect square, that is, if exist $z\in\mathbb{N}$ such that $r^{2}+t^{2}=z^{2},$ then $K$ must be of the form $K=\frac{\sqrt{2}\alpha}{m}$ such that $G.C.D(\alpha,m)=1$, with $\frac{1}{\sqrt{2}}<\frac{\alpha}{m}\leq 1$, becuase $(1\leq K \leq\sqrt{2})$. Then in (\ref{eq*21}) we get
\begin{eqnarray}\label{eq*22}
t=\frac{\alpha}{m}z
\end{eqnarray}
Therefore, it is clear that $z=\tau m$ for some $\tau \in \mathbb{N}$, so in $r^{2}+t^{2}=z^{2}$
\begin{eqnarray}\label{eq*23}
r^{2}+\alpha^{2}\tau^{2}=\tau^{2}m^{2},
\end{eqnarray}
This implies that $r$ and $t$ have the common factor $\tau$, which is a contradiction, because $M.D.C(t;r)=1$. Therefore, $r^{2}+t^{2}$ cannot be a perfect square.
Hence by (\ref{eq*21}), the expression of $r^{2}+t^{2}$ should be as follows
\begin{eqnarray}\label{eq*24}
r^{2}+t^{2}=2K^{2a}w^{2}.
\end{eqnarray}
for some $a,w\in\mathbb{N}$. Then replacing (\ref{eq*24}) in (\ref{eq*21}), we have
\begin{eqnarray}\label{eq*25}
t=\dfrac{\sqrt{2}}{2}K^{a +1}w \sqrt{2}=wK^{a+1}
\end{eqnarray}
from there $a=2b+1$ because $K^{2}$ is a rational number, for some $b\in\mathbb{N}$. So, we get $t=wK^{2(b+1)}$, with $K^{2}=\dfrac{2}{\lambda^{2}+1}$ and $\lambda=\frac{r}{t}$, the
\begin{eqnarray*}\label{eq*26}
t=\dfrac{w.2^{b+1}.t^{2b+2}}{\left(r^{2}+t^{2} \right)^{b+1}},
\end{eqnarray*}
hence
\begin{eqnarray*}\label{eq*27}
(r^{2}+t^{2})^{b+1}=w\cdot 2^{b+1}\cdot t^{2b+1}.
\end{eqnarray*}
This last relationship is a contradiction, since $G.C.D(r,t)=1$ and $t$ is even and $r$ is odd.

Therefore, this proves that the Fermat equation $A^{n}+B^{n}=C^{n}$ has no solution if $n$ is odd.
\end{proof}

\begin{center} \subsection*{Acknowledgment} \end{center}
The authors thank God for allowing this work to be carried out and completed.

The second author was supported and was funded by CONCYTEC-FONDECYT within the framework of the call
“Proyecto Investigación Básica 2019-01” [380-2019-FONDECYT].
 

\end{document}